\documentclass[12pt,oneside]{amsart}
\usepackage{amssymb}
\usepackage{verbatim}
\usepackage{xcolor}
\usepackage{graphicx}

\usepackage[top=1in, left=1.1in, right=1.1in, bottom=1in]{geometry}

\newtheorem{theorem}{Theorem}[section]
\newtheorem{proposition}[theorem]{Proposition}
\newtheorem{lemma}[theorem]{Lemma}
\newtheorem{corollary}[theorem]{Corollary}
\newtheorem{question}[theorem]{Question}

\newtheorem*{GPR}{The Generalized Property R Conjecture}
\newtheorem*{wGPR}{The Weak Generalized Property R Conjecture}
\newtheorem*{sGPR}{The Stable Generalized Property R Conjecture}

\newtheorem{conjecture}[theorem]{Conjecture}

\theoremstyle{definition}
\newtheorem{definition}[theorem]{Definition}

\makeatletter
\newtheorem*{rep@theorem}{\rep@title}
\newcommand{\newreptheorem}[2]{%
\newenvironment{rep#1}[1]{%
 \def\rep@title{#2 \ref{##1}}%
 \begin{rep@theorem}}%
 {\end{rep@theorem}}}
\makeatother

\newreptheorem{theorem}{Theorem}
\newreptheorem{corollary}{Corollary}

\theoremstyle{remark}
\newtheorem{remark}[theorem]{Remark}

\numberwithin{equation}{section}

\newcommand{\Ss}{\mathcal{S}}
\newcommand{\T}{\mathcal{T}}
\newcommand{\Hh}{\mathcal{H}}
\newcommand{\Ll}{\mathcal{L}}
\newcommand{\D}{\mathcal{D}}

\newcommand{\Z}{\mathbb{Z}}

\newcommand{\CP}{\mathbb{CP}}

\begin{document}

\title{Classification of trisections and the Generalized Property R Conjecture}

\author{Jeffrey Meier}
\address{Department of Mathematics, Indiana University, 
Bloomington, IN 47408}
\email{jlmeier@indiana.edu}
\urladdr{http://pages.iu.edu/~jlmeier/} 

\author{Trent Schirmer}
\address{Department of Mathematics, Oklahoma State University, 
Stillwater, OK 74078}
\email{trent.schirmer@okstate.edu}
\urladdr{www.trentschirmer.com} 

\author{Alexander Zupan}
\address{Department of Mathematics, University of Texas at Austin, Austin, TX 78712}
\email{zupan@math.utexas.edu}
\urladdr{http://math.utexas.edu/users/zupan} 

\begin{abstract}
We show that the members of a large class of unbalanced four-manifold trisections are standard, and we present a family of trisections that is likely to include non-standard trisections of the four-sphere. As an application, we prove a stable version of the Generalized Property R Conjecture for $c$--component links with tunnel number at most $c$.
\end{abstract}

\maketitle

\section{Introduction}\label{sec:intro}

A trisection is a decomposition of a four-manifold into three standard pieces.  The theory of trisections, due to Gay and Kirby \cite{gay-kirby:trisections}, provides a four-dimensional analogue to the theory of Heegaard splittings for three-manifolds and promises a new bridge between well-established techniques and results from three-manifold theory and the less well-understood realm of four-manifolds.  

Recall that an \emph{$n$--dimensional 1--handlebody of genus $g$} is a space diffeomorphic to $\natural^g(S^1\times D^n)$. The following definition of a trisection is slightly more general than the original one appearing in \cite{gay-kirby:trisections}.

\begin{definition}\label{def:tri}
Let $X$ be a closed, connected, orientable, smooth four-manifold. A \emph{$(g;k_1,k_2,k_3)$--trisection} of $X$ is a quadruple $(\Sigma,X_1,X_2,X_3)$ satisfying the following conditions:
\begin{itemize}
	\item $X=X_1\cup X_2\cup X_3$;
	\item $X_i$ is four-dimensional 1--handlebody of genus $k_i$ for $i\in\{1,2,3\}$;
	\item $X_i\cap X_j$ is a three-dimensional handlebody of genus $g$ for $i\neq j$; and
	\item $\Sigma=X_1\cap X_2\cap X_3$ is a closed, orientable surface of genus $g$.
\end{itemize}
The \emph{genus} of the trisection is the genus $g$ of $\Sigma$.
\end{definition}

The original definition in \cite{gay-kirby:trisections} requires that $k_1=k_2=k_3$; here, we will relax this condition. If $k_1=k_2=k_3$, we call the trisection \emph{balanced}; otherwise, it is \emph{unbalanced}.

The simplest example of a trisection is the genus zero trisection of $S^4$, which is a decomposition of $S^4$ into three four-balls which are glued pairwise along three-ball halves of their boundaries.  This is the unique trisection of genus zero, and there are precisely three balanced trisections of genus one.  These trisections correspond to the manifolds $\CP^2$, $\overline\CP^2$, and $S^1\times S^3$, and their diagrams are shown in Figure \ref{fig:GenusOneBal} below.  However, as we will see, there are three more \emph{unbalanced} trisection of genus one, which correspond to the three stabilizations of the genus zero trisection of $S^4$.

In addition to deducing an existence result for trisections, Gay and Kirby define a stabilization operation and prove that any two trisections of a fixed four-manifold have a common stabilization.  Their stabilization operation is effectively the combination of three smaller stabilization operations, arranged to preserve balance.  Since we are interested in unbalanced trisections, we will restrict our attention to the three component operations instead of the balanced operation. (See Section 3.)

The main result of this paper is that a large class of unbalanced trisections is trivial in the sense that each member can be expressed as the connected sum of (unbalanced) genus one trisections.  In particular, they are stabilizations of connected sums of balanced genus one trisections.

\begin{theorem}\label{thm:class}
	Suppose that $X$ admits a $(g;k_1,k_2,k_3)$--trisection $\T$ with $k_1\geq g-1$, and let $k' = \max\{k_2,k_3\}$.  Then, $X$ is diffeomorphic either to $\#^{k'}(S^1\times S^3)$ or to the connected sum of $\#^{k'}S^1\times S^3$ with one of $\CP^2$ or $\overline\CP^2$, and $\T$ is the connected sum of genus one trisections.
\end{theorem}

The proof makes essential use of ``Heegaard-Kirby diagrams'', which we introduce in Section \ref{sec:Heeg-Kirb}, and which let us pass easily between handle decompositions and trisections of four-manifolds.  In particular, they allow us to utilize deep classification results for Dehn surgeries between connected sums of $S^1\times S^2$ (drawing on multiple works \cite{gabai:II, gordon:combinatorial,scharlemann:reducible}) in order to prove the trisection classification result Theorem \ref{thm:class}. We then use Heegaard-Kirby diagrams to go back in the opposite direction, leveraging Theorem \ref{thm:class} to prove the following Dehn surgery result.

\begin{corollary}\label{coro:links}
	Suppose that $L$ is a $c$--component link in $\#^k(S^1\times S^2)$ with an integral framed surgery to $\#^{c+k}(S^1\times S^2)$.
\begin{enumerate}
\item If $L$ has tunnel number $c+k-1$, then $L$ is a $c$--component 0--framed unlink.
\item If $L$ has tunnel number $c+k$, then there is a sequence of handleslides taking the split union of $L$ with a $(k+1)$--component 0--framed unlink to a $(c+k+1)$--component 0--framed unlink.
\end{enumerate}
\end{corollary}

As a special case of this corollary, we see that $c$--component links with tunnel number $c-1$ satisfy the Generalized Property R Conjecture, while $c$--component links with tunnel number $c$ satisfy a slight weakening of this conjecture.  We discuss the connection between trisections and the Generalized Property R Conjecture in Section \ref{sec:propR}, where we introduce a family of trisections of $S^4$ that seem highly likely to be non-standard.

\subsection*{Conventions}

Unless otherwise mentioned, we work in the smooth category throughout.  A manifold $Y$ is said to be \emph{properly embedded} in $X$ if it is transverse to $\partial X$ and $Y\cap \partial X=\partial Y$.  A \emph{proper isotopy} of $Y$ in $X$ is a homotopy of $Y$ in $X$ through proper embeddings.  Let $N(Y,X)$ denote a closed regular neighborhood of $Y$ in $X$, and let $E(Y,X)=\overline{X\setminus N(Y,X)}$ be the \emph{exterior} of $Y$ in $X$. We shall often drop the ambient space $X$ from this notation when no confusion can arise.

\subsection*{Acknowledgements}

We would like to thank David Gay and Marty Scharlemann for helpful conversations.  We would also like to thank Cameron Gordon and Nicholas Zufelt for making us aware of Theorem 6.9 of \cite{gordon:combinatorial}. The second author is grateful to Charles Frohman for sitting down with him in Iowa City to look at the Gay-Kirby paper for the first time. Finally, we thank Utah State University Moab and the organizers of the Moab Topology Conference, during which the current version of this paper was born. The first author was supported by NSF grant DMS-1400543.  The third author was supported by NSF grant DMS-1203988.

\section{Heegaard splittings}\label{sec:Heegs}

In this section, we formally introduce Heegaard splittings of compact, orientable three-manifolds.  We recommend \cite{scharlemann:handbook} for a more detailed introduction.

A \emph{Heegaard splitting} of a closed, orientable three-manifold $Y$ is a decomposition of $Y$ into a pair of three-dimensional handlebodies.  This definition can be generalized to compact $Y$ by using compression bodies.

\begin{definition}\label{def:compbods}

Let $\Sigma$ be a closed surface.  A \emph{compression body} is a compact three-manifold $H$ constructed from $\Sigma \times I$ by attaching 2--handles to $\Sigma \times \{0\}$ and capping off any resulting 2-sphere boundary components with 3--balls.  We define $\partial_+ H = \Sigma \times \{1\}$ and $\partial_- H = \partial H \setminus \partial_+ H$.

\end{definition}

\begin{definition}
Let $Y$ be a compact, connected, orientable three-manifold.  A \emph{genus $g$  Heegaard splitting} of $Y$ is a triple $(\Sigma,H_1,H_2)$ satisfying the following conditions:

\begin{itemize}
\item $Y=H_1\cup H_2$;
\item $H_i$ is a compression body for $i=1,2$; and 
\item $\Sigma=H_1\cap H_2=\partial_+ H_1=\partial_+ H_2$ is a closed, orientable surface of genus $g$.

\end{itemize}

The surface $\Sigma$ is said to be a \emph{Heegaard surface} for $Y$.
\end{definition}

Notice our definition forces $\partial Y=\partial_-H_1\cup \partial_-H_2$.  In the case that $Y$ is closed, this agrees with the above definition in which $H_1$ and $H_2$ are handlebodies.  The following result is well-known.

\begin{theorem}
Every compact, orientable, connected three-manifold admits a Heegaard splitting.
\end{theorem}

One way to describe a Heegaard splitting of a closed three-manifold is by means of a Heegaard diagram, which is defined using compressing disks in each of the handlebodies in the splitting.

\begin{definition}

Suppose $H$ is a genus $g$ handlebody and $\mathcal{D}=D_1\cup \cdots \cup D_g$ is a disjoint union of properly embedded disks in $H$.  The collection $\mathcal{D}$ is said to be \emph{complete} if $E(\mathcal{D}, H)$ is homeomorphic to a ball.  In this case, the curves $\partial\mathcal{D}$ in $\partial H$ are said to be a \emph{defining set of curves} for $H$.
\end{definition}

\begin{definition}

A \emph{Heegaard diagram} of a Heegaard splitting $(\Sigma,H_\alpha,H_\beta)$ is a triple $(\Sigma,\alpha,\beta)$, where $\alpha$ is a defining set of curves for $H_\alpha$ in $\Sigma$ and $\beta$ is a defining set of curves for $H_\beta$ in $\Sigma$. 

\end{definition}

Although every Heegaard diagram determines a single Heegaard splitting, every Heegaard splitting admits infinitely many distinct Heegaard diagrams.  The simplest Heegaard diagrams and their corresponding splittings arise frequently in this paper.

\begin{definition}
A Heegaard diagram $(\Sigma,\alpha,\beta)$ in a genus $g$ surface $\Sigma$ is said to be \emph{$(g,k)$--standard} if the $\alpha$ and $\beta$ curves can be indexed so that
\begin{itemize}
\item $\alpha_i=\beta_i$ for all $i\leq k$.
\item $|\alpha_i\cap \beta_j|=\delta_{ij}$ for all $i>k$.
\end{itemize}
A Heegaard diagram is called \emph{standard} if it is $(g,k)$--standard for some $0\leq k\leq g$.
\end{definition}

A $(g,k)$--standard Heegaard diagram defines a genus $g$ Heegaard splitting of $\#^k(S^1\times S^2)$ for all $g\geq k\geq 0$ (throughout, we define $\#^0(S^1\times S^2)=S^3$).  In fact this is the unique genus $g$ splitting of $\#^k(S^1\times S^2)$ up to isotopy according to the following classical result of Waldhausen.

\begin{theorem}\label{thm:waldhausen}\cite{waldhausen} Suppose $Y\cong\#^k(S^1\times S^2)$.  Then for all $g\geq k\geq 0$, there is a unique genus $g$ Heegaard splitting of $Y$ up to isotopy.

\end{theorem}

The connected sum operation extends naturally to Heegaard splittings.  Let $\Hh'$ and $\Hh''$ be Heegaard splittings for three-manifolds $Y'$ and $Y''$ given by $(\Sigma',H_\alpha',H_\beta')$ and $(\Sigma'',H_\alpha'',H_\beta'')$, respectively.  Then there is a natural Heegaard splitting $\Hh$ for $Y'\#Y''$ given by $(\Sigma,H_\alpha,H_\beta)$ where
$$(\Sigma,H_\alpha,H_\beta) = (\Sigma' \# \Sigma'',H_\alpha'\natural H_\alpha'' ,H_\beta'\natural H_\beta'').$$

Let $\delta$ denote the curve on $\Sigma$ that determines the decomposition $\Sigma=\Sigma'\#\Sigma''$.  This curve has the property that it bounds disks $D_\alpha$ and $D_\beta$ in $H_\alpha$ and $H_\beta$, respectively.

\begin{definition}
	A Heegaard splitting $\Hh$ given by $(\Sigma,H_\alpha,H_\beta)$  is \emph{reducible} if there exists a curve $\delta$ on $\Sigma$ that bounds disks $D_\alpha$ and $D_\beta$ in $H_\alpha$ and $H_\beta$, respectively.
\end{definition}

The curve $\delta$ is called a \emph{reducing} curve.  The following classical result is due to Haken. 

\begin{theorem}\label{thm:haken}\cite{haken}
	Suppose $Y$ is a reducible three-manifold with Heegaard surface $\Sigma$.  Then there is a reducing sphere $S$ for $Y$ that meets $\Sigma$ in a single reducing curve.
\end{theorem}

Consequently, every Heegaard splitting of a reducible three-manifold is reducible.  The final operation on Heegaard splittings that we will need to understand is stabilization.

\begin{definition}\label{def:heegstab}
Let $\Hh=(\Sigma,H_\alpha,H_\beta)$ be a Heegaard splitting of $M$ and let $\omega$ be a boundary parallel arc properly embedded in $H_\beta$.  Then if $H_\alpha'=H_\alpha\cup N(\omega)$, $H_\beta'=E(\omega, H_\beta)$ and $\Sigma'=H_\alpha'\cap H_\beta'$, the triple $\Hh'=(\Sigma',H_\alpha',H_\beta')$ is also a Heegaard splitting of $Y$, called the \emph{stabilization} of $\Hh$.  Conversely, we say that $\Hh$ is a \emph{destabilization} of $\Hh'$.
\end{definition}

Alternatively, the stabilization of a Heegaard splitting $\Hh$ is just $\Hh\# \Ss$, where $\Ss$ is the genus one Heegaard splitting of $S^3$.  Hence, if both $(\Sigma',H_\alpha',H_\beta')$ and $(\Sigma'',H_\alpha'',H_\beta'')$ are stabilizations of the same Heegaard splitting of $Y$, then there is an isotopy of $Y$ taking $\Sigma'$ to $\Sigma''$ and $H_\alpha'$ to $H_\alpha''$. In other words, the Heegaard splittings are isotopic.  On the other hand, destabilizations do not always exist, and when they do they are not always unique up to isotopy.

As an example, when $g>k$, the $(g,k)$--standard Heegaard splitting is stabilized, so it follows from Theorem \ref{thm:waldhausen} that every genus $g$ Heegaard splitting of $\#^k(S^1\times S^2)$ with $g>k$ is stabilized.

Two Heegaard splittings $\Hh_1$ and $\Hh_2$ for $Y$ are called \emph{stably isotopic} if there exist integers $m_1$ and $m_2$ such that the Heegaard splittings $\Hh_1'$ and $\Hh_2'$ obtained by stabilizing $\Hh_1$ and $\Hh_2$ $m_1$ times and $m_2$ times, respectively, are isotopic.  The following is a classical result of Reidemeister and Singer.

\begin{theorem}\label{thm:reid_sing} \cite{reidemeister,singer}
	Any two Heegaard splittings for a fixed three-manifold are stably isotopic.
\end{theorem}

\section{Trisections}

In this section we describe how analogues of the concepts introduced in the previous section naturally arise in the theory of trisections as well.  Gay and Kirby proved the following existence result for trisections.

\begin{theorem}\label{thm:triext} \cite{gay-kirby:trisections}
Every closed, orientable, connected, smooth four-manifold admits a balanced trisection.
\end{theorem}

Given a trisection $(\Sigma,X_1,X_2,X_3)$ of $X$ as in Definition \ref{def:tri}, let $H_\beta=X_1\cap X_2$, $H_\gamma=X_2\cap X_3$, and $H_\alpha=X_3\cap X_1$.  We call the union of the handlebodies $H_\alpha\cup H_\beta\cup H_\gamma$ the \emph{spine} of the trisection.  Note that each pair of three-dimensional handlebodies in the spine of a trisection forms a Heegaard splitting of the boundary of the four-dimensional 1--handlebody on which they lie.  For example, $(\Sigma,H_\alpha,H_\beta)$ is a Heegaard splitting of $\partial X_1\cong\#^{k_1}(S^1\times S^2)$.

On the other hand, one may thicken the spine $H_\alpha\cup H_\beta \cup H_\gamma$ to obtain a four-manifold with three boundary components, each of which is homeomorphic to some $\#^{k_i}(S^1\times S^2)$.  By \cite{laudenbach-poenaru}, there is a unique way to cap off each of these boundary components with a four-dimensional handlebody.  It follows that every trisection is determined up to diffeomorphism by its spine, which in turn is determined by a diagram of curves in $\Sigma$.

\begin{definition}
Let $X$ be a four-manifold with trisection $\T$ with spine $H_\alpha\cup H_\beta\cup H_\gamma$. A \emph{trisection diagram} for $\T$ is a quadruple $(\Sigma,\alpha,\beta,\gamma)$ where $\alpha$, $\beta$, and $\gamma$ are defining sets of curves in $\Sigma$ for $H_\alpha$, $H_\beta$, and $H_\gamma$, respectively.
\end{definition}

As with Heegaard splittings, a given trisection admits many distinct trisection diagrams, although each diagram determines a single trisection.  Using diagrams, we can readily define the class of \emph{standard} trisections.

\begin{definition} \label{def:tristand}
A trisection diagram $(\Sigma,\alpha,\beta,\gamma)$ is called \emph{standard} if each of the Heegaard diagrams $(\Sigma,\alpha,\beta)$, $(\Sigma,\alpha,\gamma)$ and $(\Sigma,\beta,\gamma)$ is standard.  A trisection is called \emph{standard} if it has a standard diagram.
\end{definition}

\begin{figure}[h!]
  \centering
    \includegraphics[width=.7\textwidth]{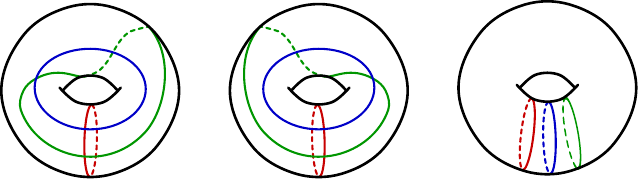}
    \caption{Standard diagrams for the three balanced genus one trisections, which correspond to $\CP^2$, $\overline\CP^2$, and $S^1\times S^3$, respectively.}
    \label{fig:GenusOneBal}
\end{figure}

It is easy to see that all six of the genus one trisections are standard.  The three possibilities in the balanced case are shown in Figure \ref{fig:GenusOneBal}, and correspond to $\CP^2$, $\overline\CP^2$, and $S^1\times S^3$.  The three unbalanced genus one trisections of $S^4$ are shown in Figure \ref{fig:GenusOneStab} and correspond to the three stabilization operations (see below).  In \cite{mz:genus2}, the authors show that all balanced genus two trisections are standard.  The main theorem of the present work states that every $(g;k_1,k_2,k_3)$--trisection with $k_1=g$ or $g-1$ is standard.  In particular, this proves that unbalanced genus two trisections are standard as well. Note that any trisection obtained by taking a connected sum of genus one trisections is standard.

\begin{figure}[h!]
  \centering
    \includegraphics[width=.7\textwidth]{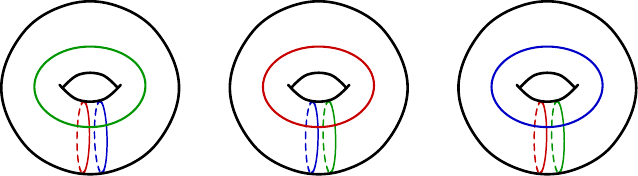}
    \caption{Standard diagrams for the three unbalanced genus one trisections, which correspond to 1--, 2--, and 3--stabilization, respectively.}
    \label{fig:GenusOneStab}
\end{figure}

We call a four-manifold $X$ \emph{reducible} if $X=X_1\#X_2$ such that neither $X_i$ is diffeomorphic to $S^4$. As with Heegaard splittings, there is a natural extension of the connected sum operation to trisections given by the formula $$(\Sigma,X_1,X_2,X_3)\#(\Sigma',X_1',X_2',X_3')= (\Sigma \# \Sigma',X_1\natural X_1',X_2\natural X_2',X_3\natural X_3').$$     There is a notion of reducibility for trisections, as well.

\begin{definition}

A trisection $\mathcal{T}$ with spine $H_\alpha\cup H_\beta\cup H_\gamma$ is \emph{reducible} if there exists a curve $\delta\subset\Sigma$ such that $\delta$ bounds disks $D_\alpha$, $D_\beta$, and $D_\gamma$ in $H_\alpha$, $H_\beta$, and $H_\gamma$, respectively. The curve $\delta$ is called a \emph{reducing curve}.

\end{definition}

An embedded three-sphere $S$ in a trisected manifold $X=X_1\cup X_2\cup X_3$ is said to be \emph{trisected reducing sphere} if $B_i=S\cap X_i$ is a three-ball for each $i$ and $\delta=S\cap \Sigma$ is an essential closed curve.

\begin{proposition}\label{prop:red_diag}
	A trisection $\T$ is reducible if and only if $\T$ admits a trisected reducing sphere.
\end{proposition}

\begin{proof}

Suppose $\T$ is reducible. Let $R_1\subset \partial X_1$ denote the two-sphere $D_\alpha\cup D_\beta$, and note that $X_1$ has a handle decomposition consisting only of a single $0$--handle and $k_1$ $1$--handles.  It is possible that the belt spheres of the $1$--handles intersect $R_1$ initially, but by \cite{laudenbach:spheres}, we can isotope and handle slide the $1$--handles over one another until their belt spheres are disjoint from $R_1$. It follows that $R_1$ can be assumed to lie in the boundary of the 0--handle such that it is disjoint form the attaching regions of the 1--handles.  Since this boundary is diffeomorphic to $S^3$, it follows that $R_1$ bounds a three--ball $B_1$ in $X_1$.

A similar argument produces three-balls $B_2$ and $B_3$ bounded by $D_\beta\cup D_\gamma$ and $D_\gamma\cup D_\alpha$, respectively. The union of these three-balls along their pairwise intersections gives the desired three-sphere $S$.  The converse is clear, since $\delta=S\cap\Sigma$ is a reducing curve.


\end{proof}

This leads us to the following question, which is the four-dimensional analogue of Haken's Lemma (Theorem \ref{thm:haken}).

\begin{question}\label{ques:Haken}
	Is every trisection of a reducible four-manifold reducible?
\end{question}

If the answer is ``yes'', it would imply that the maximum number of summands that can occur in a connected sum decomposition of a closed orientable smooth four-manifold $X$ is bounded from above by the minimum genus of a trisection of $X$.  In particular, it would imply the smooth Sch\"{o}nflies conjecture.

Finally, we define the process of stabilization.

\begin{definition}\label{def:tristab}
Let $\T=(\Sigma,X_1,X_2,X_3)$ be a trisection for $X$, and let $\omega$ be a properly embedded, boundary parallel arc in $H_{jk}=X_j\cap X_k$.  A \emph{$i$--stabilization} of $\mathcal{T}$ is a trisection $\mathcal{T}'=(\Sigma',X_1',X_2',X_3')$, where $X_i'=X_i\cup N(\omega)$, $X_j'=E(\omega,X_j)$ for $j\not=i$, and $\Sigma'=X_1'\cap X_2'\cap X_3'$. Conversely, $\mathcal{T}$ is said to be an \emph{$i$--destabilization} of $\mathcal{T}'$.
\end{definition}

If $\mathcal{T}$ is a $(g;k_1,k_2,k_3)$--trisection, then the $1$--stabilization of $\mathcal{T}$ is a $(g+1;k_1+1,k_2,k_3)$ trisection (and similarly for $i=2,3$).  Definition \ref{def:heegstab} for Heegaard splittings is similar; however, the construction is symmetric -- it does not matter whether the arc $\omega$ is properly embedded in $H_1$ or $H_2$, as the resulting Heegaard splittings are isotopic. For trisections, the result of an $i$--stabilization is generally not isotopic to the result of a $j$--stabilization when $j\neq i$. However, any two $i$--stabilizations of a trisection are isotopic.  Moreover, $i$--stabilizations and $j$--stabilizations commute; the result of performing an $i$--stabilization followed by a $j$--stabilization is isotopic to the result of performing a $j$--stabilization followed by an $i$--stabilization.

We note that our definition of stabilization differs slightly from the one appearing in \cite{gay-kirby:trisections}; Gay and Kirby define stabilization to be the result of a 1--, 2--, and 3--stabilization in order to preserve the balance of a trisection.

Just as with Heegaard splittings, there is an alternative to Definition \ref{def:tristab} that uses the connected sum operation on trisections.  Recall the three unbalanced genus one trisections of $S^4$ whose diagrams are shown in Figure \ref{fig:GenusOneStab}.

\begin{definition}
	Let $\T$ be a trisection for $X$.  Then, for any unbalanced genus one trisection $\Ss$ of $S^4$, $\T'=\T\#\Ss$ is also a trisection for $X$ and is called a \emph{stabilization} of $\T$.  Any trisection $\T$ that can be written as $\T'=\T\#\Ss$ is called \emph{stabilized}.
\end{definition}

As with reducibility, there is a diagrammatic criterion which allows us to deduce that a trisection is stabilized.

\begin{proposition}\label{prop:stab_diag}

Suppose $\T$ is a trisection with spine $H_\alpha\cup H_\beta\cup H_\gamma$, and suppose there is a pair of curves $\omega, \gamma$ properly embedded in $\Sigma$ with the following properties.

\begin{itemize}

\item $\omega$ bounds a disk in $H_\alpha$ and $H_\beta$;
\item $\gamma$ bounds a disk in $H_\gamma$; and,
\item $|\omega\cap \gamma|=1$.

\end{itemize}

Then $\T$ is $1$--stabilized.  The obvious analogues hold for $i=2,3$ as well.

\end{proposition}

\begin{proof}

This is clear when the genus of $\T$ is one.  When the genus is greater than one, $\delta=\partial N(\omega\cup \gamma)$ is an \emph{essential} curve which bounds disks in each of $H_\alpha,H_\beta,$ and $H_\gamma$.  By Proposition \ref{prop:red_diag}, $\delta$ is a reducing curve on $\Sigma$ that decomposes $\T$ into two trisections, one of which is the genus one trisection of $S^4$ corresponding to $1$--stabilization.

\end{proof}

Gay and Kirby show that the analogue of the classical Reidemeister-Singer Theorem (Theorem \ref{thm:reid_sing}) holds for trisections as well.

\begin{theorem}\cite{gay-kirby:trisections}\label{thm:tristab}
	Any two trisections of a smooth four-manifold $X$ are stably isotopic.
\end{theorem}


Aside from Theorem \ref{thm:tristab}, we know almost nothing about the set of trisections associated with a smooth manifold $X$.  Even the case of $S^4$ (equipped with its standard smooth structure) is not well-understood, although it clearly admits a family of standard trisection obtained via stabilization.  Let $\Ss^0$ denote the standard $(0,0)$--trisection of $S^4$, and let $\Ss^{k_1,k_2,k_3}$ denote the trisection obtained by performing $k_i$ $i$--stabilizations of $\Ss^0$ for $1\leq i\leq 3$.  Each of these trisections is unique up to isotopy.  Despite the fact that it is likely false (see Section \ref{sec:propR}), we boldly make the following conjecture as a four-dimensional analogue to Waldhausen's Theorem (Theorem \ref{thm:waldhausen}) for $S^4$.

\begin{conjecture}\label{ques:wald} Every trisection $\T$ of $\Ss^4$ is isotopic to $\Ss^0$ or one of its stabilizations $\Ss^{k_1,k_2,k_3}$.
\end{conjecture}

In other words, we posit that every positive genus trisection of $S^4$ stabilized. Theorem \ref{thm:class} proves Conjecture \ref{ques:wald} for the case of $(g;g-1,1,0)$--trisections. 

\begin{remark}\label{rmk:homology}
Let $X$ be a four-manifold that admits a $(g;k_1,k_2,k_3)$--trisection.  The induced handle decomposition (described in Section \ref{sec:Heeg-Kirb}) yields $\chi(X)=k_1+k_2+k_3-g+2$.  It follows that any trisection of $S^4$ must satisfy $g=k_1+k_2+k_3$.  Moreover, if $k_i=0$ for any $i$, then $\pi_1(X)\cong 1$. (The next section shows how such a trisection induces a handle decomposition of $X$ with no 1--handles.)  By work of Freedman, if $\chi(X)=2$ and $\pi_1(X)\cong 1$, then $X$ is homeomorphic to $S^4$ \cite{freedman, freedman-quinn}.
\end{remark}

\section{Heegaard-Kirby diagrams}\label{sec:Heeg-Kirb}

In this section, we discuss how to pass between handle decompositions and trisections of four-manifolds using \emph{Heegaard-Kirby diagrams}.

\begin{definition}\label{def:Heeg-Kirb}

Suppose $L$ is a $c$--component integer-framed link in $\#^n(S^1\times S^2)$ such that Dehn surgery along the framing of $L$ yields $\#^m(S^1\times S^2)$.  Suppose $(\Sigma,H_1,H_2)$ is a genus $g$ Heegaard splitting of $E(L,\#^n(S^1\times S^2))$ such that $H_1$ is a handlebody.  Then the pair $\Ll=(L,\Sigma)$ is said to be a $(g;n,c,m)$--\emph{Heegaard-Kirby diagram}. 

\end{definition}

We will regard two Heegaard-Kirby diagrams in $\#^n(S^1\times S^2)$ as equivalent if there is a diffeomorphism of $\#^n(S^1\times S^2)$ which preserves the corresponding links (with framings) and Heegaard surfaces.  The framed link $L$ of a $(g;n,c,m)$--Heegaard-Kirby diagram determines a handle decomposition of a closed four-mani\-fold $X$ which consists of a single $0$--handle, $n$ $1$--handles, $c$ $2$--handles, $m$ 3--handles, and a single $4$--handle.  The Heegaard splitting $(\Sigma,H_1,H_2)$ of Definition \ref{def:Heeg-Kirb} then determines a way to carve out a trisection from this handle decomposition.  The proof of the following proposition is contained in Lemma 14 of \cite{gay-kirby:trisections}, which uses different terminology.

\begin{proposition}\label{thm:HeegKirb1}
	Every $(g;n,c,m)$--Heegaard-Kirby diagram $\Ll$ corresponding to a four-manifold $X$ determines a unique $(g;n,g-c,m)$--trisection $\T(\Ll)$ of $X$ up to diffeomorphism. 
\end{proposition}

\begin{proof}
	Let $\Ll=(L,\Sigma)$ be a $(g;n,c,m)$--Heegaard-Kirby diagram for $X$, and let $Y'\cong\#^m(S^1\times S^2)$ denote the result of the integral surgery on $Y\cong\#^n(S^1\times S^2)$ along $L$.  Let $X_1$ and $X_3$ denote the unions of the 0--handle and the 1--handles and the 3--handles and the 4--handles, respectively.  Then, $\partial X_1 = Y$, and $\partial X_3 = Y'$.
	
	Now, let $(\Sigma,H_\alpha,H_\beta)$ be the genus $g$ Heegaard splitting of $Y$ induced by $\Sigma$, and suppose that $(\Sigma,H_\alpha, E(L,H_\beta))$ is the corresponding Heegaard splitting of $E(L,Y)$.  The result of surgering $H_\beta$ along $L$ is a new handlebody, which we denote $H_\gamma$.  Let $X_2$ be the trace of this surgery.  In other words, $X_2$ is the result of attaching $c$ 2--handles to $H_\beta\times[0,1]$ along $L\times \{1\}$.  It follows that $X_2\cong \natural^{g-c}(S^1\times B^3)$, since $H_\beta\times[0,1]\cong\natural^g(S^1\times B^3)$ and since the 2--handles attached along $L$ are dual to $c$ of the 1--handles in $\natural^g(S^1\times B^3)$.
	
	Thus, $X_1$, $X_2$, and $X_3$ are all four-dimensional 1--handlebodies of genus $n$, $g-c$, and $m$, respectively.  By construction, $X_1\cap X_2 = H_\beta$, $X_2\cap X_3 = H_\gamma$, and $X_3\cap X_1 = H_\alpha$ are all three-dimensional handlebodies of genus $g$ whose common boundary is $\Sigma$.  Therefore, $\T(\Ll)=(\Sigma,X_1,X_2,X_3)$ is a $(g;n,g-c,m)$--trisection of $X$.
\end{proof}

In \cite{gay-kirby:trisections}, Gay and Kirby use Proposition \ref{thm:HeegKirb1} to give one proof of Theorem \ref{thm:triext}.

\begin{remark}
	The roles played by $H_\alpha$, $H_\beta$, and $H_\gamma$ can be interchanged due to the inherent symmetry of trisections.  We use this fact to prove the main theorem, and this symmetry is a major strength of viewing four-manifolds through the lens of trisections.  For the sake of consistency, however, we will stick with the chosen designations throughout.
\end{remark}

Now, every curve on $\Sigma$ which bounds a disk in $H_2 = E(L,H_\beta)$ bounds a disk in both $H_\beta$ and $H_\gamma$.  This observation, in conjunction with Propositions \ref{prop:red_diag} and \ref{prop:stab_diag}, allows us to immediately deduce the following easy but powerful lemma.

\begin{lemma}\label{lem:Heeg-Kirb}
Let $\Ll=(L,\Sigma)$ be a Heegaard-Kirby diagram for $X$.  Let $\Hh=(\Sigma,H_1,H_2)$ be a Heegaard splitting of $E(L)$, and suppose that $\Ll$ induces a trisection $\T(\Ll)$ of $X$.

\begin{itemize}
	\item If $\Hh$ is reducible, then $\T$ is reducible.
	\item If $\Hh$ is stabilized, then $\T$ is $2$--stabilized.
\end{itemize}
\end{lemma}

On the other hand, given a trisection $\T=(\Sigma,X_1,X_2,X_3)$ of $X$, we can find many Heegaard-Kirby diagrams $(L,\Sigma)$ which induce $\T$. As we will see, each such diagram arises from a special choice of disks in $H_\beta$, which amounts to changing our framed link description of $X$.

\begin{definition}
Let $(\Sigma,H,H')$ be a genus $g$ Heegaard splitting for a three-manifold $Y$.  A disk $D$ properly embedded in $H$ is \emph{primitive in $H$ with respect to $H'$} if there is a disk $D'$ properly embedded in $H'$ such that $|D\cap D'|=1$. If $Y\cong\#^k(S^1\times S^2)$, then a \emph{complete collection of primitive disks} for $H$ is any disjoint union $D_1\cup\cdots \cup D_{g-k}$ of nonisotopic disks which are primitive in $H$ with respect to $H'$.
\end{definition}

Note that Theorem \ref{thm:waldhausen} implies that for every Heegaard splitting $(\Sigma,H,H')$ for $\#^k(S^1 \times S^2)$, there is a complete collection of primitive disks for $H$.

Just as a Heegaard splitting has many Heegaard diagrams, a trisection has many Heegaard-Kirby diagrams.  Given a $(g;k_1,k_2,k_3)$--trisection $(\Sigma,X_1,X_2,X_3)$ of $X$ with spine $H_\alpha\cup H_\beta\cup H_\gamma$, let $L=L_1\cup \cdots L_{g-k_2}$ be the boundary of a complete collection of primitive disks for $H_\gamma$ with respect to $H_\beta$. Then after pushing $L$ into $H_\beta$ and giving it the surface framing induced by $\Sigma$, we obtain a $(g;k_1,g-k_2,k_3)$--Heegaard-Kirby diagram $\Ll(\D)=(L,\Sigma)$, where $(H_\alpha, E(L,H_\beta),\Sigma)$ is a Heegaard splitting of $E(L,\partial X_1)$. A key fact to observe is that the primitivity of $L$ with respect to $H_\beta$ ensures that $E(L,H_\beta)$ will be a compression body.  This gives the following lemma.

\begin{lemma}\label{lem:Heeg-Kirb2}
For any $(g;k_1,k_2,k_3)$--trisection $\T=(\Sigma,X_1,X_2,X_3)$ of $X$ and complete collection $\D$ of disks in $H_\gamma$ that are primitive with respect to $H_\beta$, we have $\T(\Ll(\D))=\T$, where $\T(\Ll(\D))$ is the trisection induced by the Heegaard-Kirby diagram $\Ll(\D)$. 
\end{lemma}

\section{Proof of the main theorem}\label{sec:proof}

An important ingredient in the proof of the main theorem is the following classification of Dehn surgeries between connected sums of $S^1\times S^2$.  This result is implicit in the more general Theorem 6.9 of \cite{gordon:combinatorial}, where the pertinent case is attributed to \cite{gabai:II} and \cite{scharlemann:reducible}.  We sketch a proof for convenience.

\begin{theorem}\label{thm:Dehn}
	Suppose that $K$ is a knot in $\#^k(S^1\times S^2)$ and that $K$ admits a non-trivial Dehn surgery to $\#^{k'}(S^1\times S^2)$, where $k' \geq k$.  Then $k' \in \{k,k+1\}$, $K$ is unknotted and contained in a ball, and the surgery slope is $\pm1/m$ for $m\in\Z\setminus\{0\}$ or zero, respectively.
\end{theorem}

Note that the theorem implies that if $K$ is a knot in $\#^k(S^1\times S^2)$ with a surgery to $\#^{k'}(S^1\times S^2)$, then $k'\in\{k-1,k,k+1\}$, since the case $k'<k$ is just the dual surgery picture to the case $k<k'$.  It follows that the above theorem completely classifies such surgeries.

\begin{proof}
	Knots in $\#^k(S^1\times S^2)$ are determined by their complements \cite{gabai:II}, so the result follows in the case that surgery on $K$ yields $\#^k(S^1\times S^2)$.
	
	Suppose instead that $K$ has a surgery to $\#^{k'}(S^1\times S^2)$, where $k' > k$.  Since $K$ is a knot, surgery on $K$ can increase the rank of the first homology by at most one, and thus $k' = k+1$.  We proceed by induction.  The case when $k=0$ is the Property R Conjecture, which was settled by Gabai \cite{gabai:III}.  Assume that the theorem is true for any $k<n$, and suppose that $K$ is a knot in $\#^n(S^1\times S^2)$ with a surgery to $\#^{n+1}(S^1\times S^2)$.  Because the rank of the first homology is increasing from $n$ to $n+1$, $K$ must be null-homologous.
	
	Suppose that $E(K)$ contains a reducing sphere $P$.  Then, by surgering $P$, we reduce the set-up to a surgery on $K$ in $\#^{n-1}(S^1\times S^2)$, so we are done by our inductive hypothesis. If $E(K)$ is $\partial$--reducible, the $E(K)$ is reducible; surgering $\partial E(K)$ along the reducing disk gives an essential sphere in $E(K)$.
	
	It follows that $E(K)$ is irreducible and $\partial$--irreducible. But now we satisfy the hypotheses of the main result of \cite{scharlemann:reducible}, which tells us that $n=0$, completing the proof.
\end{proof}

We are now prepared to prove our main result.

\begin{reptheorem}{thm:class}
	Suppose that $X$ admits a $(g;k_1,k_2,k_3)$--trisection $\T$ with $k_1\geq g-1$, and let $k' = \max\{k_2,k_3\}$.  Then, $X$ is diffeomorphic either to $\#^{k'}(S^1\times S^3)$ or to the connected sum of $\#^{k'}S^1\times S^3$ with one of $\CP^2$ or $\overline\CP^2$, and $\T$ is the connected sum of genus one trisections.
\end{reptheorem}

\begin{proof}

The case $k_1=g$ is trivial, since it implies $X$ has a handle decomposition without 2--handles and must be the connected sum of $k_2 = k_3$ copies of $S^1 \times S^3$.

For the case $k_1=g-1$, we proceed by induction on $g$.  If $g=1$, the statement follows from the classification of genus one trisections.  Suppose that the theorem holds whenever $g'<g$.  It suffices to show that $\T$ is reducible, because this implies that $\T$ can be written as a connected sum of a $(g';k_1',k_2',k_3')$--trisection $\T'$ and a $(g'';k_1'',k_2'',k_3'')$--trisection $\T''$, where $g = g' + g''$ and $g-1 = k_1 = k_1' + k_1''$.  By the inductive hypothesis, both $\T'$ and $\T''$ satisfy the conclusions of the theorem.

To match the conventions in Section \ref{sec:Heeg-Kirb}, regard $\T$ as a $(g;k_2,g-1,k_3)$--trisection from now on. Given such a trisection, a complete collection of primitive disks for $H_\gamma$ with respect to $H_\beta$ will consist of a single disk, and will thus induce a Heegaard-Kirby diagram $(L,\Sigma)$ in $\#^{k_2}(S^1\times S^2)$ in which $L$ has a single component.  If necessary, permute $X_2$ and $X_3$ so that $k_3 \geq k_2$.  By Theorem \ref{thm:Dehn}, the knot $L$ is an \emph{unknot} in $\#^{k_2}(S^1\times S^2)$ and $k_3 \in \{k_2,k_2+1\}$.  

Therefore, we have $E(L,\#^{k_2}(S^1\times S^2))\cong V\# (\#^{k_2}(S^1\times S^2))$, where $V$ denotes the solid torus.  Hence, if $k_2>0$, then $E(L,\#^{k_2}(S^1\times S^2))$ is reducible.  In the case $k_2=0$, $E(L)$ is a solid torus.  It is a standard theorem (see \cite{lei:stability} for a reference) that any genus $g\geq 2$ Heegaard splitting of the solid torus is reducible (in fact stabilized).  In any case, the splitting of $E(L)$ is reducible, and thus by Lemmas \ref{lem:Heeg-Kirb} and \ref{lem:Heeg-Kirb2}, the trisection $\T$ is reducible as well.

\end{proof}

\begin{remark}\label{rmk:parameters}
	
Notice that if $\T$ is a $(g;g,k_2,k_3)$--trisection, then $k_2 = k_3$.  Similarly, if $\T$ is a $(g;g-1,k_2,k_3)$--trisection, then we have that $k_3\in\{k_2-1,k_2,k_2+1\}$.  In addition, if a manifold $X$ with a $(g;g-1,k_2,k_3)$--trisection $\T$ has a $\CP^2$ or $\overline{\CP}^2$ summand, then $\T$ may be expressed as the connected sum of the standard $(1;0,0,0)$--trisection and a $(g-1;g-1,k_2,k_3)$--trisection; thus, $\T$ must have parameters $(g;g-1,k_3,k_3)$.
\end{remark}

\section{Applications to the Generalized Property R Conjecture}\label{sec:propR}

In this section, we prove Corollary \ref{coro:links}, and we discuss the applications of our results to the Generalized Property R Conjecture and related problems, including the search for non-standard trisections of $S^4$.  

Let $L$ be a $c$--component link in a closed three-manifold $Y$.  A \emph{tunnel system} $\tau$ for $L$ is a collection of arcs $t_1,\ldots,t_n$ that are embedded in $Y$ such that $t_i\cap L=\partial t_i\cap L$ for all $i$.  A tunnel system is called \emph{complete} if $E(L\cup\tau,Y)$ is a handlebody. The minimum number of arcs in a complete tunnel system for $L$ is called the \emph{tunnel number} of $L$, and is denoted $\tau(L)$.  Note that $L$ admits a complete tunnel system with $n$ tunnels if and only if $E(L)$ has a Heegaard splitting $(\Sigma,H_1,H_2)$ of genus $n+1$ such that $H_1$ is a handlebody.

The proof of Theorem \ref{thm:class} above relies heavily on the Dehn surgery classification given by Theorem \ref{thm:Dehn}.  Using the symmetry inherent to trisections, we may now take Theorem \ref{thm:class} and extract a new surgery result, stated in the following corollary.

\begin{repcorollary}{coro:links}
	Suppose that $L$ is a $c$--component link in $\#^k(S^1\times S^2)$ with an integral framed surgery to $\#^{c+k}(S^1\times S^2)$.
\begin{enumerate}
\item If $L$ has tunnel number $c+k-1$, then $L$ is a $c$--component 0--framed unlink.
\item If $L$ has tunnel number $c+k$, then there is a sequence of handleslides taking the split union of $L$ with a $(k+1)$--component 0--framed unlink to a $(c+k+1)$--component 0--framed unlink.
\end{enumerate}
\end{repcorollary}

\begin{proof}
	We may assume that $c \geq 2$, since the case $c=1$ is covered by Theorem \ref{thm:Dehn}.  Let $\tau= t_1 \cup \cdots\cup t_n$ be a complete tunnel system for $L$, and let $X$ be the four-manifold constructed by capping off the boundary components of the trace of the surgery on $L$ with four-dimensional 1--handlebodies.  First, if $n = c+k-1$, then $E(L,\#^k(S^1 \times S^2))$ has a genus $c+k$ Heegaard splitting $\Sigma$, and $(\Sigma,L)$ is a $(c+k;k,c,c+k)$--Heegaard-Kirby diagram for $X$ inducing a $(c+k;k,k,c+k)$--trisection $\T$.  By the correspondence in Section \ref{sec:Heeg-Kirb}, we know that $L$ lies on the trisection surface $\Sigma$ and bounds disks in $H_\gamma$.  However, since $H_{\alpha} \cup H_{\gamma}$ is a genus $(c+k)$ Heegaard splitting for $\#^{c+k}(S^1 \times S^2)$, every curve bounding a disk in $H_{\gamma}$ also bounds a disk in $H_{\alpha}$.  We conclude that $L$ is a $c$--component unlink in $\#^k(S^1 \times S^2) = H_{\alpha} \cup H_{\beta}$.
	
	If $n = c+k$, then there is a $(c+k+1;k,c,c+k)$--Heegaard-Kirby diagram $(\Sigma,L)$ for $X$ which yields a $(c+k+1;k,k+1,c+k)$--trisection $\T$.  By Theorem \ref{thm:class}, the trisection $\T$ may be written as a connected sum of genus one trisections, and by Remark \ref{rmk:parameters}, $X$ does not have a $\CP^2$ or $\overline{\CP}^2$ summand.  It follows that each summand of $\T$ is a genus one diagram for $S^4$ or $S^1 \times S^3$, and $\T$ has a diagram $(\Sigma,\alpha',\beta',\gamma')$ such that each curve in $\gamma'$ bounds a disk in at least one of $H_{\alpha}$ or $H_{\beta}$.  This implies that $\gamma'$ is a $(c+k+1)$--component 0--framed unlink in $\#^k(S^1 \times S^2) = H_{\alpha} \cup H_{\beta}$.  Let $\D_{\gamma'}$ denote the collection of disks in $H_{\gamma}$ bounded by $\gamma'$.
	
	As above, the link $L$ bounds a complete collection $\D$ of primitive disks in $H_\gamma$ with respect to $H_\beta$.  Since $\D$ is primitive with respect to $H_{\beta}$ and $(\Sigma,H_\beta,H_\gamma)$ is a Heegaard splitting of $\#^{k+1}(S^1\times S^2)$, we can extend $\D$ to a complete collection $\overline\D = \D \cup \D^*$ of disks for $H_{\gamma}$ such that the $k+1$ curves $U = \partial \D^*$ bound disks in both $H_{\beta}$ and $H_{\gamma}$.  It follows that $U$ is a $(k+1)$--component 0--framed unlink.  Finally, by \cite{johannson}, the collections $\overline\D$ and $\D_{\gamma'}$ of disks are related by a sequence of handleslides, and thus their boundaries, $L\cup U$ and $\gamma'$, are similarly related.
	
\end{proof}

We note that in the case that $k=0$ and $c=2$, conclusion (1) above has been established in \cite{gst} with a proof attributed to Alan Reid.  This proof can be extended to show (1) holds for $k=0$ and for any value of $c$.  The restriction to links in $S^3$ is of special importance, because it relates to the Generalized Property R Conjecture (GPRC). For an overview of recent developments in the GPRC, see \cite{gst} and \cite{williams:GPR}.  The strongest statement of the conjecture is likely false.

\begin{GPR}
	Suppose that $L$ is a $c$--compo\-nent link in $S^3$ with an integral surgery to $\#^c(S^1\times S^2)$.  Then there is a sequence of handleslides transforming $L$ into a $c$--component unlink.
\end{GPR}

The authors of \cite{gst} also present a weaker version of the conjecture.

\begin{wGPR}
		Suppose that $L$ is a $c$--component link in $S^3$ with an integral surgery to $\#^c(S^1\times S^2)$.  Then, for some integers $r$ and $s$, the union $L'$ of $L$ and a distant, zero-framed unlink of $r$ components and a collection or $s$ canceling Hopf pairs has the property that a sequence of handleslides transforms $L'$ into the split union of a $(c+r)$--component unlink and a collection of $s$ canceling Hopf pairs.
\end{wGPR}

We propose a third, related conjecture to serve as an intermediary between the GPRC and the Weak GPRC.  This new conjecture is more closely related to Corollary \ref{coro:links} and to trisections than the two conjectures appearing in \cite{gst}.

\begin{sGPR}
	Suppose that $L$ is a $c$--component link in $S^3$ with an integral surgery to $\#^c(S^1\times S^2)$.  Then, after possibly introducing a distant, zero-framed unlink $U$ of $r$ components, there is a sequence of handleslides transforming $L\cup U$ into the $(c+r)$--component unlink.
\end{sGPR}

Observe that Corollary \ref{coro:links} implies that $c$--component links with tunnel number at most $c$ satisfy the Stable GPRC.  Clearly, the GPRC implies the Stable GPRC, and the Stable GPRC implies the Weak GRPC.  In \cite{gst}, the GPRC for two-component links is discussed in detail, and an infinite family $\{L_{n,1}\}$ of likely counterexamples is given, along with an interesting relationship between the GPRC and the Andrews-Curtis Conjecture.  Roughly, the latter conjecture states that every balanced presentation of the trivial group can be changed to a trivial presentation via Andrews-Curtis moves, which include the inversion of a relator, replacing a relator with a product of two relators, and conjugating a relator by a generator.  If a presentation $P$ of the trivial group can be transformed to the trivial presentation in this way, then $P$ is \emph{Andrew-Curtis trivial}.  See \cite{andrews-curtis} for more details.

In \cite{gst}, the authors establish the following theorem.
\begin{theorem}\label{thm:AC}\cite{gst}
	If $L_{n,1}$ satisfies the Generalized Property R Conjecture, then the following presentation of the trivial group is Andrews-Curtis trivial:
	$$P_n = \langle x, y \,|\, yxy=xyx, x^{n+1}=y^n\rangle.$$
\end{theorem}

For $n\geq 3$, the presentations $P_n$ are believed to be Andrews-Curtis nontrivial. Although they conclude that the GPRC is likely false, the authors of \cite{gst} show that the links $\{L_{n,1}\}$ still satisfy the Weak GPRC.  Thus, the following question is of interest.

\begin{question}\label{stableln}
Do the links $\{L_{n,1}\}$ satisfy the Stable Generalized Property R Conjecture?
\end{question}

There is a well-known weakening of the Andrews-Curtis Conjecture, called the \emph{Stable Andrews-Curtis Conjecture}, which asserts that any balanced presentation of the trivial group may be reduced to the trivial presentation with Andrews-Curtis moves and a stabilization/destabilization move in which a trivial generator-relator pair is either added or deleted.  Although the stable version is typically distinguished from the original conjecture in group theory literature, the distinction is often blurred among topologists (see, for example, \cite{gompf-stipsicz}).  It is unknown whether the above presentations $P_n$ are stably Andrews-Curtis trivial.

Let $L$ be a link, and let $(\Sigma,L)$ be a Heegaard-Kirby diagram which gives rise to a trisection $\T$, with components labeled as above.  We leave it to the reader to verify that stabilizing $\Sigma$ and adding a distant 0--framed unknot to $L$ can be arranged to correspond to 3--stabilizing the trisection $\T$, stabilizing $\Sigma$ without modifying $L$ corresponds to 2--stabilizing $\T$ (Lemma \ref{lem:Heeg-Kirb}), and adding a canceling Hopf pair may be arranged to correspond to 1--stabilizing $\T$.  

For some $n\geq 3$, consider $L_{n,1}$, one of the proposed counterexamples, and let $\Sigma$ be a Heegaard surface for $E(L_{n,1})$.  Let $\T_n$ denote the trisection induced by $(L_{n,1}, \Sigma)$.  In \cite{gst}, the authors show that the link $L'_{n,1}$ consisting of $L_{n,1}$ and a single Hopf pair is handleslide equivalent to a 0--framed unlink and a single Hopf pair, and the four-manifold $X$ described by the surgery is a standard $S^4$.  This fact does not immediately imply that a 1--stabilization of $\T_n$ is trivial (since handleslides may not take place in the surface $\Sigma)$, but it is reasonable to believe that applying a 1--stabilization to $\T_n$ and then stabilizing a relatively small number (possibly zero) of additional times will give one of the standard trisections of $S^4$.

On the other hand, if $\T_n$ itself is standard, then it follows that there is an $r$--component 0--framed unlink $U_r$ such that $L_{n,1}\sqcup U_r$ is handle slide equivalent to a $(2+r)$--component unlink, so the links $L_{n,1}$ have the Stable Generalized Property R.  We conclude with the following.

\begin{proposition}
	If the trisection $\T_n$ is standard, then the presentation $P_n$ is stably Andrews-Curtis trivial.  More generally, if every trisection of the smooth four-sphere is standard, then the presentations $P_n$ are stably Andrews-Curtis trivial.
\end{proposition}

Note that, as mentioned above, only 1--stabilization results in the introduction of a Hopf pair to the surgery description.  It follows performing 2-- or 3--stabilizations on $\T_n$ will not change the conclusion of the above proposition.  In particular, \emph{any} choice of Heegaard surface $\Sigma$ for $E(L)$ is allowed.

Therefore, if Question \ref{stableln} has a negative answer, then not only are the trisections $\T_n$ non-standard splittings of $S^4$, but it would also follow that any trisection in an infinite sequence of 2-- and 3--stabilizations performed on $\T_n$ is also non-standard!  We consider this possibility to be both striking and probable.

\bibliographystyle{acm}
\bibliography{CSUTBiblio.bib}

\end{document}